\documentclass[12pt]{amsart} 
\usepackage[utf8]{inputenc}
\usepackage[T1]{fontenc}
\usepackage{lmodern}
\usepackage{amsmath, amsthm, amssymb, amsfonts}
\usepackage{latexsym}
\usepackage{amsxtra} 
\usepackage[all]{xy}
\usepackage{nicefrac,mathtools}
\usepackage[shortlabels, inline]{enumitem}
\usepackage{microtype}
\usepackage{hyperref}
\usepackage{graphicx,calc}
\newlength\myheight
\newlength\mydepth
\settototalheight\myheight{Xygp}
\usepackage{xcolor}
\usepackage{tikz-cd}
\usepackage[lite]{amsrefs}
\usepackage[T1]{fontenc}
\usepackage[utf8]{inputenc}
\usepackage{thmtools}
\usepackage[color=Lavender]{todonotes}
\usepackage{autonum}

\usepackage{geometry}
\numberwithin{equation}{section}
\theoremstyle{plain}
\newtheorem{thm}[equation]{Theorem}
\newtheorem{lem}[equation]{Lemma}

\theoremstyle{definition}
\newtheorem{defn}[equation]{Definition}

\theoremstyle{remark}\newtheorem{rem}[equation]{Remark}
\theoremstyle{remark}
 \newtheorem*{rem*}{Remark}
\newtheorem*{remarks*}{Remark} 

\newcommand{\R}{\mathbb{R}}

\newcommand{\Z}{\mathbb{Z}}

\newcommand{\Q}{\mathbb{Q}}

\newcommand{\tb}{\ensuremath{\mathrm{tb}}}
\newcommand{\rot}{\ensuremath{\mathrm{rot}}}
\newcommand{\link}{\ensuremath{\mathrm{lk}}}
\newcommand{\order}{\ensuremath{\mathrm{o}}}
\newcommand{\sln}{\ensuremath{\mathrm{sl}}}
\newcommand{\p}{\ensuremath{\partial}}
\newcommand{\A}{\ensuremath{\mathcal{A}}}
\newcommand{\std}{\mathrm{std}}

\title[Non-Simple knots in contact 3-manifolds]{Non-Simple knots in contact 3-manifolds}

\author[Ipsita Datta]{Ipsita Datta $^{1}$}
\address{$^1$Department of Mathematics ETH Z\"{u}rich, Switzerland; \tiny{ipsita.datta@math.ethz.ch}}
\author[Tanushree Shah]{Tanushree Shah $^{2}$}
\address{$^2$Chennai Mathematical Institute, India; \tiny{tanushrees@cmi.ac.in}}

\usepackage[foot]{amsaddr}

 \thanks{\emph{Subjclass[2020]}: 57K10, 57K14, 57K33 }

\begin{document}
\begin{abstract} 
We present new families of examples of non-simple prime Legendrian and transversal knots in tight Lens spaces, which demonstrate that the botany of Legendrians in Lens space is rich. In fact, there are more non-isotopic Legendrians that are topologically isotopic to the $n$-twist knot in a Lens space $L(\alpha, \beta)$, than in $S^3$. We also include connect sum formulas for rational variants of classical invariants, $\tb_\Q$, $\rot_\Q$, and $\sln_\Q$, which indicate that prime knots are the right playground to look for exotic behaviour.
\end{abstract}

\maketitle

\section{Introduction}
Consider a contact $3$-manifold $(M^3, \xi)$. A knot $\Lambda \subset M$ is called {\bf Legendrian} if $\Lambda$ is everywhere tangent to the contact structure $\xi$. Two Legendrians are considered equivalent if they are isotopic through Legendrian knots. The characteristics of the contact structure play a significant role in the classification of Legendrian representatives for a given knot. In some cases, the classical invariants, namely the Thurston-Bennequin number (\tb) and the rotation number (\rot), can completely classify Legendrian representatives up to Legendrian isotopy. In this case, the knot type is referred to as {\bf Legendrian simple} \cite{EF}. 

However, for certain knot types, these classical invariants are insufficient for distinguishing between Legendrian representatives. By defining a combinatorially computable version of Legendrian Contact Homology, Chekanov showed that there exist non-simple Legendrian knots in $S^3$, that is, there exist two Legendrians $\Lambda_1$ and $\Lambda_2$ in $S^3$ that have the same classical invariants $\tb$ and $\rot$, but they are not Legendrian isotopic (see \cite{C1}, \cite{C2}). 
Sabloff extended the definition of a combinatorial LCH to $S^1$-bundles \cite{Sabloff_2003}, Licata to Lens spaces \cite{Licata_2011}, and Sabloff-Licata to Seifert Fibered spaces \cite{Licata-Sabloff_2012}, with universally tight contact structures. A combinatorial description of LCH for Legendrians on connect sums $\#^k (S^1 \times S^2)$ was provided by Ekholm-Ng in \cite{ekholm-ng}.
Sabloff-Licata also provided an example of non-simple Legendrians in Lens spaces in \cite{Licata-Sabloff_2012}. 
These examples are essentially of the type $\Lambda_i \# \Lambda_F$, $i =1,2$, where $\Lambda_i$ are the the $4$-twist knot version of Chekanov's examples, and $\Lambda_F$ is a Legendrian unknot in the lens space $L(\alpha, \beta)$ that is homotopic to a trivial fiber (when $L(\alpha, \beta)$ is expressed as a Seifert fibered space) but wraps around non-trivially around the singular fiber, see Figure~\ref{fig:prime decomposition}. They explicitly work out the $(-5)$-twist knot case.

A knot is called {\bf prime} if it is a non-trivial knot that cannot be written as the connected sum of two non-trivial knots. Notice that the examples in \cite{Licata-Sabloff_2012} are not prime.

The main objective of this paper is to present new examples of non-simple knots in Lens spaces. We give examples of prime non-simple Legendrian knots in Lens spaces that are topologically $(-n)$-twist knots for every integer $n \geq 3$.
\begin{thm}\label{thm:prime legendrian examples}
In each lens space $L(\alpha, \beta)$ with a universally tight contact structure $\xi$, there exist non-isotopic Legendrian prime knots that represent $(-n)$-twist knots, $n\geq 3$, 
but have the same rational classical invariants. There exists one such Legendrian knot for every partition of $n = z+s+2$ for $0 \leq z\leq n-2$.
\end{thm}
We also expand on the Licata--Sabloff examples and, following their argument closely, provide examples of non-prime non-simple Legendrians in Lens spaces $L(p,q)$ that are topologically $(-n)$-twist knots for every integer $n \geq 3$.
\begin{thm}\label{thm:licata-sabloff legendrian examples}
In each lens space $L(\alpha, \beta)$ with a universally tight contact structure $\xi$, there exist non-Legendrian isotopic knots, distinct from those in Theorem~\ref{thm:prime legendrian examples}, that represent $(-n)$-twist knots, $n\geq 3$, but have the same rational classical invariants. There exists one such Legendrian knot for every partition of $n = l + (n-l)$ for $1 \leq l \leq \lceil n/2\rceil$.
\end{thm}
Notice that the number of non-isotopic prime Legendrians in Theorem~\ref{thm:prime legendrian examples} is higher than the corresponding number in $S^3$ (see \cite{etnyre-ng-vertesi}) and in Theorem~\ref{thm:licata-sabloff legendrian examples}, see Remark~\ref{rem: more knots in lens spaces}.
This is an interesting indication that Legendrians can pick up information about the global contact structure.

 A smooth knot $K \subset M$ is called a {\bf transverse knot} if it is everywhere transverse to the contact planes, i.e.
  \[
  T_pK \oplus \xi_p = T_pM \quad \text{for all } p \in K.
  \]
Two transverse knots are said to be {\bf transversely isotopic} if there exists a smooth isotopy through transverse knots taking one to the other.
A topological knot type $\mathcal{K}$ is called {\bf transversely simple} if transverse isotopy classes of $\mathcal{K}$ are classified uniquely by their classical invariant, the self-linking number $\operatorname{sl}$. Else, it is called {\bf transversely non-simple}. 

  We present large classes of Legendrian and transversely non-isotopic cables. We use the fact that simplicity for transverse knots can be studied by studying the corresponding Legendrian representatives.

\begin{thm}\label{thm:cable of torus knot}
    Consider the lens space $L(\alpha, \beta)$ with a universally tight contact structure $\xi$. There exist cable knots that are both Legendrian and transversely non-simple.
\end{thm}

\begin{thm}\label{thm:cable of twist knot}
    Consider a Lens space $L(\alpha, \beta)$ with a universally tight contact structure $\xi$. There exist cables of $(-n)$-twist knots that are Legendrian non-simple.
\end{thm}

Contact invariants usually behave well under connect sums. We prove rational analogues of connect sum formulas \cite{Etnyre-Honda_2003_connect_sums} for classical invariants $\tb$, $\rot$ and self-linking numbers using first principles and surgery formulas from \cite{kegel_2018}.
\begin{thm}\label{thm: tb under connect sum}
    Consider Legendrian knots $\Lambda_i \subset (M_i, \xi_i)$. Then (rational) Thurston-Bennequin number of the connect sum $\Lambda_1 \# \Lambda_2 \subset M_1 \# M_2$ is given by
    \begin{align}
        \tb_\Q (\Lambda_1 \# \Lambda_2) = \tb_\Q(\Lambda_1) + \tb_\Q(\Lambda_2) + 1.
    \end{align}
     Consider a choice of rational Seifert surfaces $\Sigma_i$ for $\Lambda_i$. The rational rotation number satisfies
    \begin{align}
        \rot_\Q(\Lambda_1\#\Lambda_2, \Sigma_1\#\Sigma_2) = \rot_\Q(\Lambda_1, \Sigma_1) + \rot_\Q(\Lambda_2, \Sigma_2).
    \end{align}
\end{thm}

\begin{thm}\label{thm: self-linking number under connect sum}
    Consider transverse knots $K_i \subset (M_i, \xi_i)$. Then the (rational) self-linking number of the connect sum $K_1 \# K_2 \subset M_1\# M_2$ is given by
    \begin{align}
        \sln_\Q (K_1 \# K_2) = \sln_\Q (K_1) + \sln_\Q(K_2) + 1.
    \end{align}
\end{thm}

We refer to a decomposition the form $\Lambda = \Lambda_1 \# \dots \Lambda_n$ of a Legendrian $\Lambda \subset (M, \xi)$ into prime pieces $\Lambda_i \subset (M_i, \xi_i)$, where $M = M_1 \# \dots \# M_n$, as a {\bf prime decomposition}.
Etnyre and Honda showed in \cite{Etnyre-Honda_2003_connect_sums} that the Legendrian prime decomposition of a Legendrian knot is unique up to moving stabilizations from one component to another and possible permutations of components. 
Additionally, in \cite{ekholm-etnyre-sullivan_2005}, it is shown that the DGA associated to Legendrian knots in $S^3$ behaves ``predictably'' under connect sums. We expect similar results to hold for the (low-energy) LCH when we consider connect sums of Legendrians in other Seifert fibered spaces. This indicates that any ``exotic'' examples of Legendrians cannot come from looking at connect sums and emphasizes the importance of the examples in Theorem~\ref{thm:prime legendrian examples} and Theorem~\ref{thm:cable of torus knot}.

\subsection*{Acknowledgements} 
The authors thank John Etnyre, Marc Kegel, Joan Licata, and Joshua Sabloff for valuable discussions.
The authors thank the organisers, Jyothisman Bhowmick and Samik Basu, of the National Centre for Mathematics workshop titled ``Holomorphic and Topological Methods in Symplectic Geometry (2024),'' as this project began when the authors met at the workshop.  
Currently, the second author receives partial support from the Infosys Fellowship.

\section{Review of construction of lens spaces}
In this section, we review the construction of Lens spaces and other Seifert fibered spaces via contact Dehn surgery. Roughly speaking, one cuts out a solid torus that is a Weinstein neighbourhood of a  Legendrian knot from $(S^3, \xi_\std)$ and glue in another such solid torus in a different way to obtain a new $3$-manifold. One can do this iteratively, or along a Legendrian link, to obtain a Seifert fibered space.

\begin{defn}
Let $K \subset M^3$ be a smooth knot in a closed $3$-manifold. Let $r$ on $\partial(\nu K)$ be a non-trivial simple closed curve and a diffeomorphism 
\begin{align}
    \phi: \partial(S^1 \times D^2) \to \partial(\nu K) \text{ which takes the meridian } \mu_0 := \{\mathrm{pt} \times S^1\} \mapsto r.
\end{align}
Then define 
\begin{align}
    M_K(r) := S^1 \times D^2 \sqcup (M \setminus {\nu K}^\circ)/ \sim,\quad \partial(S^1 \times D^2) \ni p \sim \phi(p) \in \partial(\nu K).
\end{align}
We say that $M_K(r)$ is obtained out of $M$ by {\bf Dehn surgery} along $K$ with {\bf slope} $r$.
\end{defn}

For a Legendrian $\Lambda \subset (M^3, \xi)$, there is an embedded copy of tubular neighbourhood $S^1 \times D^2 \subset (J^1(S^1), \xi_\std
)$ of the zero section where the zero section is identified with $\Lambda$. Let us refer to this neighbourhood as a {\bf Weinstein torus} around $\Lambda$ and denote it by $\nu \Lambda$.

Then there is a distinguished longitude, referred to as the {\bf contact longitude} $\lambda_C$ on $\partial (\nu \Lambda)$, given by pushing $\Lambda$ in a direction transverse to the contact planes, for example, in the direction of the Reeb vector field. The {\bf meridian} $\mu$ refers to a simple closed curve on $\partial (\nu \Lambda)$ of the form $\{\mathrm{pt}\} \times \partial D$.

If we write the Dehn surgery slope as
\begin{align}
     r = \alpha\mu + \beta \lambda_c,
 \end{align}
the topological type of the surgered manifold $M_\Lambda(r)$ is determined by the {\bf contact surgery coefficient} $r_c = \alpha/\beta \in \Q \cup \{\infty\}$.
\begin{thm}[\cite{kegel_2018}]
    The surgered manifold $M_\Lambda(r)$ carries a (non-unique) contact structure $\xi_\Lambda(r)$, which coincides with the old contact structure $\xi$ on $M\setminus \nu \Lambda^\circ$.
\end{thm}
The contact manifold $(M_\Lambda(r), \xi_\Lambda(r))$ is said to be obtained from $(M, \xi)$ by {\bf contact Dehn surgery} along the Legendrian knot $\Lambda$ with slope $r$. 

We consider $S^3$ with the co-oriented standard tight contact structure, $\xi_{\std}$. 
A Lens space $L(\alpha, \beta)$ can be obtained by contact surgery on $(S^3, \xi_\std)$ along a Legendrian unknot with slope $r = \alpha/\beta$. The following theorem says that the contact structure on $L(\alpha, \beta)$ can be chosen to be universally tight.
\begin{thm}[\cite{honda_2000_classification_tight_1}]
 There are exactly two tight contact structures on $L(\alpha, \beta)$ with $\beta \neq p-1$ which are universally tight. If $\beta = \alpha - 1$, there is exactly one.   
\end{thm} 
A Lens space has a natural fibration over $S^2$ with one exceptional fiber and is a $3$-manifold with a semi-free $S^1$-action. 

One can iterate this surgery or perform a surgery along a Legendrian link to obtain a more general Seifert fibered surface. One can also begin with a more general $3$-manifold, namely, the manifold obtained by a slope $(1,b)$ surgery on $\Sigma_g  \times S^1$ for any closed genus $g$ surface $\Sigma_g$. The resulting space is a Seifert fibered space with Seifert invariants $(g,b;(\alpha_1, \beta_1), \dots, (\alpha_r, \beta_r))$. Note that $(1,1)$-surgery on $S^2 \times S^1$ gives $S^3$. Note that any Seifert fibered space carries a semi-free $S^1$-action.

The {\bf rational Euler number} of a Seifert fibered space $M$ with Seifert invariants $(g,b;(\alpha_1, \beta_1), \dots, (\alpha_r, \beta_r))$ is 
\begin{align}
    e(M) = -b - \sum_{i=1}^r \frac{\beta_i}{\alpha_i}.
\end{align}
So, for a lens space $L(\alpha, \beta)$
\begin{align}
    e(L(\alpha, \beta)) = -1 -\frac{\beta}{\alpha}.
\end{align}
The (low-energy) Legendrian contact homology (LCH) defined in \cite{Licata-Sabloff_2013} uses an $S^1$-invariant transverse contact structure on the Seifert fibered spaces. The following theorem puts a condition on the existence of such contact structures. In particular, every lens space for $0 < \alpha, 0 < \beta < \alpha -1$ supports such a structure. It is clear that such a structure is universally tight.
\begin{thm}[\cites{kamishima-tsuboi_1991, lisca-matic_2004}]
    On a Seifert fibered space, there exists an $S^1$-invariant transverse contact form if and only if the rational Euler number is negative.
\end{thm}


\section{Classical rational invariants under connect sum}\label{sec:classical invariants under connect sum}
Consider a contact manifold $(M^3, \xi)$ with a tight contact structure. In this section, we compute connect sum formulas for the rational versions of classical invariants. These hold for any tight contact structures, that is, both for universally tight and virtually overtwisted.

A knot $K \subset M$ is said to be {\bf rationally null homologous} if there exists $\order \in \Z$ such that $\order[K] = 0 \in H_1(M;\Z)$, that is, $[K] = 0 \in H_1(M; \Q)$. In this case, we refer to $\order$ as the {\bf order} of $K$. Sometimes, we will use a subscript $\order_M$ if the manifold is important or relevant.

For rationally nullhomologous Legendrian knots in contact $3$-manifolds, one can
generalize the classical invariants to the so-called  rational classical invariants
$\tb_\Q$, $\rot_\Q$ and $\sln_\Q$ (see, for example, \cite{BE}, \cite{baker-grigsby_2009}, \cite{GO}). We first review these definitions.

For a Legendrian knot $\Lambda \in (M, \xi)$, consider a Weinstein torus $\nu \Lambda \subset M$ that is contactomorphic to a small neighbourhood of the $0$-section in the $1$-jet space $(J^1 (\Lambda), \xi_{\std})$. Let the {\bf meridian} $\mu \in H_1(\partial \mu \Lambda)$ denote the class represented by a simple closed curve on $\partial \nu \Lambda$ that is contractible in $\nu \Lambda$.
A {\bf (rational) Seifert longitude} of an oriented rationally nullhomologous knot $\Lambda$
of order $\order$ is a class $r \in H_1(\p \nu K)$ such that
\begin{equation}
    \mu \cdot r = \order \text{ and } r= 0 \text{ in } H_1(M \setminus \nu \Lambda).
\end{equation}
A rational Seifert surface for an oriented rationally nullhomologous knot $\Lambda$ is a
surface $S$ with boundary in $M \setminus \Lambda$ whose boundary represents a (rational)
Seifert longitude $\lambda_S$ of $\Lambda$.
 The Seifert longitude of a rationally nullhomologous knot is unique \cite{durst-kegel-klukas_2016}.   

    The {\bf rational Thurston–Bennequin invariant} of a rationally nullhomologous Legendrian knot $\Lambda$ is defined as
    \begin{align}
        \tb_\Q(\Lambda) = \frac{1}{\order}(\lambda_C \cdot \lambda_S),
    \end{align}
where $\lambda_C$ denotes the contact longitude and the intersection product $\cdot$ is taken in $H_1(M \setminus \nu \Lambda).$

The {\bf rational rotation number} of a rationally nullhomologous Legendrian $\Lambda \subset (M. \xi)$ is given by
    \begin{align}
        \rot_\Q(\Lambda, \Sigma) = \frac{1}{\order}\langle e(\xi, \Lambda), [\Sigma]\rangle = \frac{1}{\order} PD(e(\xi, \Lambda))\cdot[\Sigma]
    \end{align}
    where $e(\xi, \Lambda)$ is the relative Euler class of the contact structure $\xi$ relative to the trivialization given by the positive tangent vector field along the knot $\Lambda$ and $[\Sigma]$ the relative homology class represented by the surface $\Sigma$, and the intersection is taken in $H_1(\partial \nu \Lambda)$.
This trivialization is equivalent to the Reeb push-off within the Darboux ball. The $\rot_\Q$ depends on the choice of Seifert surface \cite{BE}.

The {\bf rational self-linking number} of a smooth knot $K$ is defined as the intersection number
    \begin{align}
        \sln_\Q(K, [\Sigma]) = \frac{1}{\order}[\Sigma]\cdot K'
    \end{align}
    where $\Sigma$ is a rational Seifert surface of $K$ of order $\order$ and $K'$ is a pushoff of $K$ in the direction of a non-vanishing section of $\xi|_\Sigma$.

Connected sum operations can be done in the contact world by choosing Darboux balls $B_i \subset (M_i, \xi_i)$ and taking connected sum such that $B_1 \# B_2$ is again a Darboux ball. By \cite{colin_1997}, tightness is preserved under connect sums. One can define Legendrian connect sums extending the definition of knot connect sums, again by working within a Darboux ball. Figure~\ref{fig:local model for Legendrian connect sum} gives a local model for Legendrian connect sum in terms of the cusps.
By \cite{Etnyre-Honda_2003_connect_sums}, the Legendrian connect sum $\Lambda_1 \# \Lambda_2$ does not depend (up to Legendrian isotopy) on the choice of cusps, balls $B_i$ and the identifying function. 

\begin{figure}[h]
    \centering
    \includegraphics[scale=0.7]{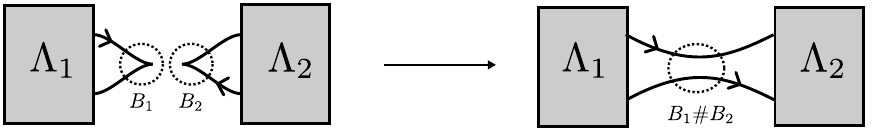}
    \caption{Local model of the connected sum of $\Lambda_1$ and $\Lambda_2$ in the front projection.}
    \label{fig:local model for Legendrian connect sum}
\end{figure}
\begin{lem}
The connect sum $K := K_1 \# K_2$ of two rationally null homologous knots is rationally null homologous in the connect sum manifold $M := M_1 \# M_2$. In particular,
\begin{align}
    \order_M(K) = \mathrm{lcm}(\order_{M_1}(K_1), \order_{M_2}(K_2)).
\end{align}
\end{lem}
This follows from the Mayer-Vietoris theorem. 

\begin{proof}[Proof of Theorem~\ref{thm: tb under connect sum}]
    We first prove the $\tb_\Q$ formula. We provide two different proofs here, catering to two ways of looking at the connected sum $M_1 \# M_2$. 
    
    The first perspective involves examining the connected sum in a local Darboux chart, see Figure~\ref{fig:tb connect sum}. As portrayed in Figure~\ref{fig:local model for Legendrian connect sum}, there exist Darboux balls $B_i \subset M_i$ with cusps $C_i$ of $\Lambda_i$. As we can take the connected sum $\Lambda_1 \# \Lambda_2$, the Darboux balls have to be combined so that one of the cusps is an up cusp and the other a down cusp. So, before surgery, the total contribution of these two cusps to the $\tb$ is $-1$. To elaborate, let us denote by $\Lambda'_i$ a pushoff of $\Lambda_i$ is a direction transverse to $\xi_i$. We can assume that within the Darboux balls, this pushoff looks like a translation in the positive $z$-direction. Then the cusps within the Darboux balls each contribute $-1/2$ to the rational linking number. To see this, notice that we may assume within the Darboux ball the rational Seifert surface $\Sigma_i$ of $\Lambda_i$, which we are using to compute the $\tb$, is simply $r_i$ copies of the $x,z$-plane region bounded by $C_i$ in the interior. 
    \begin{figure}[h]
        \centering
        \includegraphics[width=0.7\linewidth]{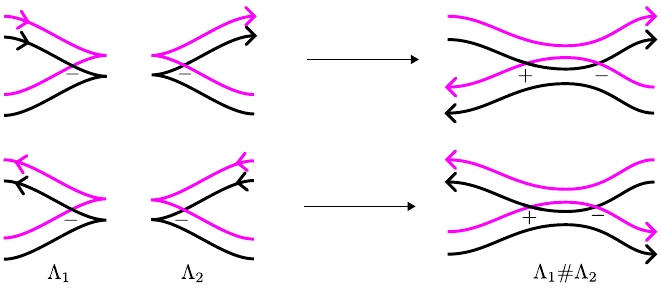}
        \caption{Thurston-Bennequin number in the local model. The transverse pushoffs are in red.}
        \label{fig:tb connect sum}
    \end{figure}So, the contribution to the linking number from within the Darboux ball is 
    \begin{align}
        \frac{1}{r_i^2}\frac{ r_i^2}{2}(\# +\mathrm{ve \, crossings} - \# -\mathrm{ve \, crossings}) =  -\frac{1}{2}.
    \end{align}
    Now we can expand the rational $\tb$'s as
    \begin{align}
        \tb_\Q(\Lambda_1) &= \link_\Q(\Lambda_1, \Lambda'_1) = \frac{1}{2}(-1) + A_1,\\
        \tb_\Q(\Lambda_2) &= \link_\Q(\Lambda_2, \Lambda'_2) = \frac{1}{2}(-1) + A_2,
    \end{align}
    where $A_i = \link_\Q(\Lambda_i, \Lambda'_i) + \frac{1}{2}$ for $i =1,2$ denotes everything in the linking number coming from outside the Darboux balls. Now, we may conclude by observing that the connect sum removes the two cusps within the Darboux ball and leaves everything else untouched. So,
    \begin{align}
        \tb_\Q(\Lambda_1 \# \Lambda_2) & = A_1 + A_2 = \tb_Q(\Lambda_1) + \tb_\Q(\Lambda_2) + 1.
    \end{align}
    
    The second perspective is through surgery diagrams. We can assume that $M_i$ is obtained from $S^3$ by Dehn surgery along links $\Lambda^i = \Lambda^i_1 \sqcup \Lambda^i_2 \sqcup \dots \sqcup \Lambda^i_{n_i}$ with topological surgery coefficients $r_j = \alpha_j/\beta_j$, for $j = 1, \dots, n_i$. Then the connected sum $M_1 \# M_2$ is obtained from $S^3$ by Dehn surgery along the link $\Lambda = \Lambda^1 \sqcup \Lambda^2$ with the same topological surgery coefficients as before. 

    We want to use the formula for rational Thurston-Bennequin invariant from \cite{kegel_2018}, which we briefly describe now. Let $\Lambda_0$ be a Legendrian link in $(S^3 \setminus \nu^\circ \Lambda, \xi_{st}) \subset (S^3, \xi_{st})$. Let $M$, which is obtained from $S^3$ by Dehn surgery along link $\Lambda = \Lambda_1 \sqcup \dots \sqcup \Lambda_n$ with topological surgery coefficients $r_j = \alpha_j/\beta_j$, have a contact structure $\xi$ that coincides with $\xi_{st}$ on $S^3$ outside the neighbourhood of $\Lambda$ identified with the normal bundle $\nu(\Lambda)$. Let $\tb_{old}$ denote the Thurston-Bennequin number of $\Lambda_0$ in $(S^3, \xi_{\std})$ and $\tb_{\Q, new}$ the rational $\tb$ in $(M, \xi)$. Then, the two are related by 
    \begin{align}
        \tb_{\Q, new} = \tb_{old} - \frac{1}{\order}\sum_{i=1}^n a_i \beta_i l_{i0} = \tb_{old} - \frac{1}{\order}\langle \mathbf{a}, \mathbf{\beta l} \rangle,
    \end{align}
    where $\order$ is the order of $\Lambda_0$ in $M$ where it is rationally null homologous, $l_{jk} := \link(\Lambda_j, \Lambda_k)$, and  $\mathbf{\beta l} = (\beta_1l_{10}, \dots, \beta_nl_{n0})$. The numbers $a_i$ are given as follows. Set, 
    \begin{align}
        Q := \begin{pmatrix} \alpha_1 & \beta_2l_{12} & \cdots &\beta_nl_{1n}\\
\beta_1l_{21} & \alpha_2 & &\\
\vdots &&\ddots& \\
\beta_1 l_{n1} &&& \alpha_n
        \end{pmatrix} \text{ and } \mathbf{l} := \begin{pmatrix}
            l_{01} \\ \vdots \\ l_{0n}
        \end{pmatrix}.
    \end{align}
    Then $L_0$ is rationally null homologous in $M$ if and only if there exists $\mathbf{a} := (a_1, \dots, a_n) \in \Z^n$ such that $\order\mathbf{l} = Q\mathbf{a}$ \cite[Lemma 6.1]{kegel_2018}.

    Going back to our connected sum, let us write $\Lambda_i = \Lambda_0^i \subset S^3 \setminus \nu^\circ \Lambda^i$ following the above notation. Then the connect sum $\Lambda_0 = \Lambda_0^1\# \Lambda_0^2 \subset S^3 \setminus \nu^\circ \Lambda$. Let us denote all numbers associated with $\Lambda^i$ by adding a superscript $i$. Keeping that in mind, note that
    \begin{align}
        l_{0j}^i := \link(\Lambda_0^i, \Lambda^i_j) = \link (\Lambda_0, \Lambda^i_j), \quad \text{ and so, }\quad
        \mathbf{l} = (
            \mathbf{l}^1, \mathbf{l}^2). 
    \end{align}
    The links $\Lambda^1$ and $\Lambda^2$ are not linked in $M$, and so, the matrix $Q$ decomposes
    \begin{align}
        Q = \begin{pmatrix}
            Q^1 & 0\\
            0 & Q^2
        \end{pmatrix}.
    \end{align}
    This implies that 
    \begin{align}
        \mathbf{a} = \order Q^{-1} \mathbf{l} = \order\left(            (Q^1)^{-1} \mathbf{l}^1, (Q^2)^{-1} \mathbf{l}^2 \right) = \left(
            \frac{\order}{\order_1}\mathbf{a}^1,\frac{\order}{\order_2}\mathbf{a}^2 \right).
    \end{align}
    Now we compute and expand the $\tb$ of the connect sum to obtain the required formula.
    \begin{align}
        \tb_{\Q, new, M} (\Lambda_0) & = \tb_{old} (\Lambda_0) - \frac{1}{\order}\langle \mathbf{a}, \mathbf{\beta l}\rangle \\
        & = \tb_{old}(\Lambda_0^1) + \tb_{old}(\Lambda_0^2) + 1 - \frac{1}{\order} \left\langle \left(
            \frac{\order}{\order_1}\mathbf{a}^1, \frac{\order}{\order_2}\mathbf{a}^2
        \right), \left(
            \mathbf{\beta l}^1, \mathbf{\beta l}^2
        \right) \right\rangle\\
        & = \tb_{old}(\Lambda_0^1) + \tb_{old}(\Lambda_0^2) + 1 - \frac{1}{d_1}\langle \mathbf{a}^1, \mathbf{\beta l}^1 \rangle - \frac{1}{d_2}\langle \mathbf{a}^2, \mathbf{\beta l}^2 \rangle\\
        & = \tb_{\Q, new, M_1} (\Lambda_0^1) + \tb_{\Q, new, M_2} (\Lambda_0^2) + 1.
    \end{align}

We now prove the rotation number formula. We again provide two different proofs.
Suppose $\Sigma_i$ for $i = 1,2$ are Seifert surfaces for $\Lambda_i$ of order $\order_i$. Then we have rotation numbers
\begin{align}
    \rot_i = \frac{1}{\order_i}\langle e(\xi, \Lambda_1), [\Sigma_i]\rangle.
\end{align}
We can assume that $\Sigma_i \subset M\setminus(\nu\Lambda_1 \cup \nu\Lambda_2)$ and can glue the Seifert surfaces $\Sigma_i$ to get a rational Seifert surface 
\begin{align}
    [\Sigma] = \frac{\order}{\order_1}[\Sigma_1] + \frac{\order}{\order_2}[\Sigma_2]
\end{align}
    for the connect sum Legendrian $\Lambda_1 \# \Lambda_2$. Assuming that the trivializations match with the Reeb direction $\partial_z$ in the Darboux ball, we can ``glue'' together the trivializations over the $\Sigma_i$'s by identifying them in the Darboux ball. Additionally, we can make sure that in a small neighbourhood of the ``neck'' of the connect sum, there is no contribution to $\langle e(\xi, \Lambda_i), [\Sigma_i]\rangle = PD(e(\xi, L_i))\cdot[\Sigma_i]$. Then
    \begin{align}
        \langle e(\xi, \Lambda_1\# \Lambda_2), [\Sigma]\rangle = \langle e(\xi, \Lambda_1), \frac{\order}{\order_1}[\Sigma_1]\rangle + \langle e(\xi, \Lambda_2), \frac{\order}{\order_2}[\Sigma_2]\rangle, 
    \end{align}
    which gives us the required result.

    For the alternate proof, we recall the formula \cite[Lemma 4.5.7]{kegel_2018}
    \begin{align}
        \rot_{\Q, new, \hat\Sigma} = \rot_{old} - \frac{1}{\order}\sum_{i=1}^k a_i \beta_i \rot
    \end{align}
    where $a_i$ and $\beta_i$ are as earlier, and $\rot_i$ denotes the rotation number of $L_i$ in $M$. Let $\mathbf{\beta\rot} = (\beta_1\rot_{1}, \dots, \beta_n \rot_{n})$. As before, let us denote $\Lambda_i = \Lambda_0^i \subset S^3\setminus \nu^\circ \Lambda^i$ and all the numbers associated to $\Lambda^i$ with an $i$ in the superscript.
    The result follows from the following computation of the rotation number of the connect sum
    \begin{align}
        \rot_{\Q, new, M}(\Lambda_0) & = \rot_{old}(\Lambda_0) - \frac{1}{\order}\langle \mathbf{a}, \mathbf{\beta\rot}\rangle\\
        & = \rot_{old}(\Lambda^1_0) + \rot_{old}(\Lambda^2_0) - \frac{1}{\order} \langle\left(\frac{\order}{\order_1}\mathbf{a^1}, \frac{\order}{\order_2}\mathbf{a^2} \right), \left(\mathbf{\beta^1\rot^1}, \mathbf{\beta^2\rot^2}\right)\rangle\\
        &= \rot_{\Q, new, M_1}(\Lambda_0^1) + \rot_{\Q, new, M_2}(\Lambda_0^2).
    \end{align}
    
\end{proof}

\begin{rem}
    Even though, within the small Darboux ball around a cusp, we can choose both trivializations of the normal bundle -- the one by the Seifert pushoff and the one by the transverse pushoff -- to be equal, this is not true globally. These two trivializations need not match even if the ambient manifold is $\R^3$ with the standard contact structure.
\end{rem}

\begin{proof}[Proof of Theorem~\ref{thm: self-linking number under connect sum}]
\begin{figure}[h]
    \centering
    \includegraphics[width=0.5\linewidth]{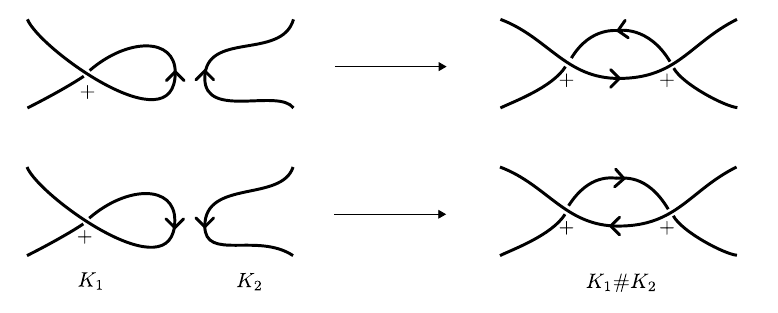}
    \caption{Self-linking number under connect sum}
    \label{fig:self linking number under connect sum}
\end{figure}
     Within a Darboux ball, we can choose the non-vanishing section of $\xi|_\Sigma$. Then, within the Darboux ball, we can track the contributions to the linking number, and the theorem follows; see Figure~\ref{fig:self linking number under connect sum}.

    We can arrive at the same conclusion using Kegel's formula to determine how the self-linking number changes under surgery. Assume $M_i$ are as in the proof of Theorem~\ref{thm: tb under connect sum} and $T_0$ is an oriented transverse knot in the complement of $\Lambda_1 \sqcup \Lambda_2$. Then, for $T_0$ rationally nullhomologous of order $\order$ in $M$ 
    the new self-linking number $\sln_{\Q, new, \hat\Sigma}$ in $(M, \xi)$ in $(M, \xi)$ with respect to a special (rational) Seifert class $\hat\Sigma$ is equal to 
    \begin{align}
        \sln_{\Q, new, \hat\Sigma} = \sln_{old} - \frac{1}{\order} \sum_{i=1}^k a_i \beta_i(l_i \mp \rot_i).
    \end{align}
    The theorem follows from a computation analogous to those in the proof of Theorem~\ref{thm: tb under connect sum}.

    For transverse pushoffs, the theorem is a corollary of Theorem~\ref{thm: tb under connect sum}.
    If $T^\pm$ are $\pm$-transverse push-offs of a Legendrian $\Lambda$, then
    \begin{align}
        \sln_\Q(T^\pm, [\Sigma]) = \tb_\Q(\Lambda) \mp \rot_\Q(\Lambda, [\Sigma]).
    \end{align}
    If the transverse knots $T(\Lambda_i)$ is the transverse pushoff of $\Lambda_i$ and $T(\Lambda_1 \# \Lambda_2)$ is a transverse pushoff of $\Lambda_1 \# \Lambda_2$, then 
    \begin{align}
        \sln_\Q(T(\Lambda_1\#\Lambda_2)) & = \tb_\Q(\Lambda_1 \# \Lambda_2) \mp \rot_\Q(\Lambda_1 \# \Lambda_2)\\
        & = \tb_\Q(\Lambda_1) + \tb_\Q(\Lambda_2) + 1 \mp \rot_\Q(\Lambda_1) \mp \rot_\Q(\Lambda_2)\\
        & = \sln_\Q(T(\Lambda_1)) + \sln_\Q(T(\Lambda_2)) + 1.
    \end{align}
    Note that if we want $T(\Lambda_1) \# T(\Lambda_2)$ to be a transverse pushoff of $\Lambda_1 \# \Lambda_2$, we need $T(\Lambda_i)$ to be either both positive or both negative transverse pushoffs of $\Lambda_i$. Accordingly, $T(\Lambda_1) \# T(\Lambda_2)$ is a positive or negative pushoff of $\Lambda_1 \# \Lambda_2$.
\end{proof}

\section{Non-simple $n$-twist knots}
\label{NST}
In this section, we describe two families of non-simple Legendrian knots in lens spaces, that is, Legendrian knots that have the same rational classical invariants, but are distinguished by their Legendrian contact homologies. The ones of the first type are prime. The ones of the second type are not prime and are an expansion of the examples of torsion knots in \cite[Section~5.2]{Licata-Sabloff_2013}. Both these classes of examples can be seen in a general Seifert fibered space by doing Legendrian surgery along a link that is unlinked with the constructed knot. 

\begin{rem}\label{rem: more knots in lens spaces}
Let $E(z,s)$ denote the Legendrian $(-n)$-twist knot, $n = z=s + 2$, in $S^3$ with Lagrangian projection as in Figure~\ref{fig:prime decomposition}. In \cite{epstein-fuchs-meyer} it was shown that $E(z,s)$ is Legendrian isotopic to $E(z', s')$ if and only if the unordered pairs $\{z, s\}, \{z', s'\}$ are the same. In contrast $\Lambda(z, s)$ is Legendrian isotopic to $\Lambda(z', s')$ in $L(\alpha, \beta)$ if and only if the ordered pairs $(z, s), (z', s')$ are the same. Thus, there are more Legendrian isotopy classes in the same $n$-twist topological isotopy class in lens spaces.
\end{rem}

\subsection{Non-simple prime twist knots}\label{sec: prime twist knots}
We present in this section a proof of Theorem~\ref{thm:prime legendrian examples} by describing the construction of the knots and computing the Poincar\'e polynomials of their LCHs.

Consider an immersed curve $\Gamma$ of the type in Figure~\ref{fig:construction of Legendrian torus knot} in $S^2$. Here, assume that we see $z+1$ half twists on the left-hand side and $s$ half twists on the right-hand side, with $z+s = n-2$. We use the $(Z,S)$ notation from \cite{etnyre-ng-vertesi}. View this $S^2$ as the base of the fibration $S^3 \to S^2$, where each fiber is a Reeb orbit. Consider surgery along a Legendrian unknot $\Lambda_0$ that intersects $S^2$ at the point $p_1$ and the Weinstein torus $\nu \Lambda_0$ for surgery can be chosen such that $\nu \Lambda_0 \cap S^2 \subset R_1$ for the region $p_1 \in R_1 \subset S^2 \setminus \Gamma$. Choose $\Lambda_0$ and $\nu \Lambda_0$, by expanding the curve $\Gamma$ if required, so that the surgered manifold $S^3_\Lambda(r)$ is a lens space $L(\alpha, \beta)$ with a universally tight contact structure. Then $\Gamma$ can be lifted to a Legendrian isotopic to the $(-n)$-twist knot.
\begin{figure}[h]
    \centering
    \includegraphics[width=0.5\linewidth]{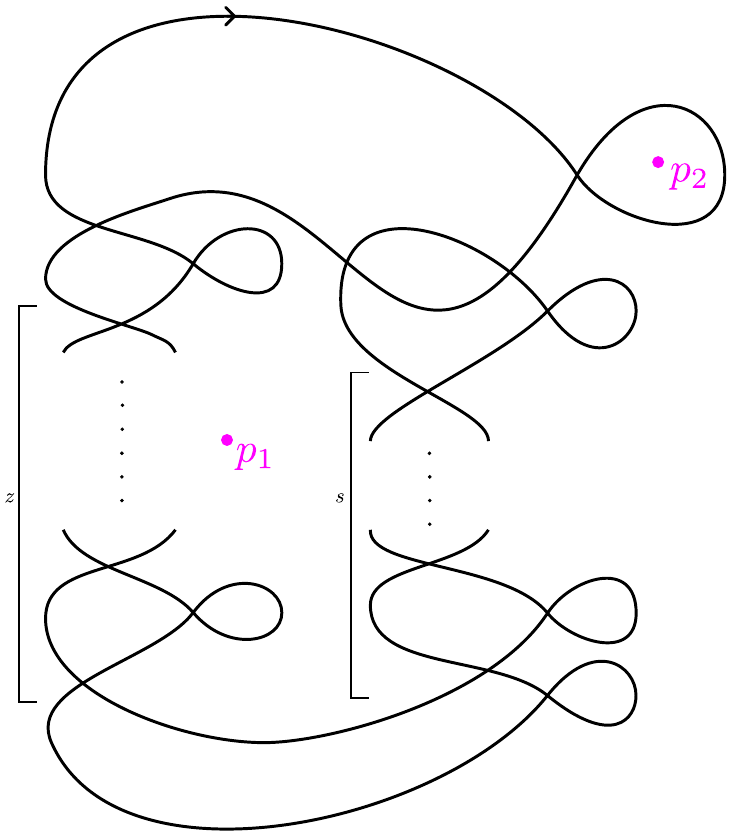}
    \caption{Construction of the $\Lambda(z,s)$ Legendrian $-n$-twist knot}
    \label{fig:construction of Legendrian torus knot}
\end{figure}
In the notation of \cite{Licata-Sabloff_2013}, this knot is represented by the labelled Lagrangian diagram that has ``defects'' given as in Figure~\ref{fig:n twist knot}.

The defect is roughly the ``error'' in bounding a disk; the boundary of a region, closed up using the $\alpha_\pm$ Reeb chords, has. See \cite[Definition~3.3]{Licata-Sabloff_2013} for the precise definition.
These knots are prime as noted in \cite{Gab}.

The classical rational invariants of these knots are (see \cite{kegel_2018})
\begin{align}
    \tb_\Q(\Lambda(z, s)) &= 1  &\text{ for } n \text{ even, }\\
    \tb_\Q(\Lambda(z, s)) &= -3 &\text{ for } n \text{ odd, and }\\
    \rot_\Q(\Lambda(z, s)) & = 0 & \text{ for all }n.
\end{align}
To see this, note that the surgery Legendrian has linking number $0$ with the Legendrian lift of $\Gamma$ in $S^3$ and use surgery formulas from \cite{kegel_2018}.

To show that $\Lambda(z,s)$ for different $(z,s)$ gives distinct Legendrian representatives of the prime $(-n)$-twist knot, we apply Chekanov's technique of linearized homology \cite{C2} to the low-energy LCH described in \cite{Licata_2011} and \cite{Licata-Sabloff_2013}.

Let $\Gamma$ denote the Lagrangian projection of the Lagrangian knot $\Lambda$. A formal capping surface of $\Gamma$ is a vector in $\Z^{|\Sigma \setminus \Gamma|}$, or equivalently, an assignment of an integer to each region of $\Sigma \setminus \Gamma$, that comes from following a variant of Seifert's algorithm \cite[Section~4.2.1.]{Licata-Sabloff_2013}.
As in {\cite[Definition~4.2]{Licata-Sabloff_2013}}, if $S = (c_1, c_2, \dots, c_n)$ is a formal capping surface, the {\bf defect} and {\bf rotation} of $S$ are the sums of the defects or rotations of the regions $R_j$, weighted by multiplicity
 \begin{equation}n(S) = \sum_j c_jn(R_j), \quad r(S) = \sum_j c_j r(R_j).\end{equation} 
 As discussed in \cite{Licata-Sabloff_2013}, the grading is only well-defined in $\Z$ modulo 
\begin{equation}2r(S) + 2\mu n(S),\end{equation}
for $S$ a formal capping surface of $\Gamma$.
Note that for all the $(-n)$-twist knots $\Lambda(z, s)$ we described, we may construct a formal capping surface $S$ such that both the total rotation $r(S) = 0$ and the defect $n(S) = 0$. Hence, the gradings are well-defined (modulo $0$).

We now describe the case of the $(-5)$-twist knots in full detail. Refer to Figure~\ref{fig: 4 twist prime knots} for this computation. 
\begin{figure}[h]
    \centering
    \includegraphics[width=\linewidth]{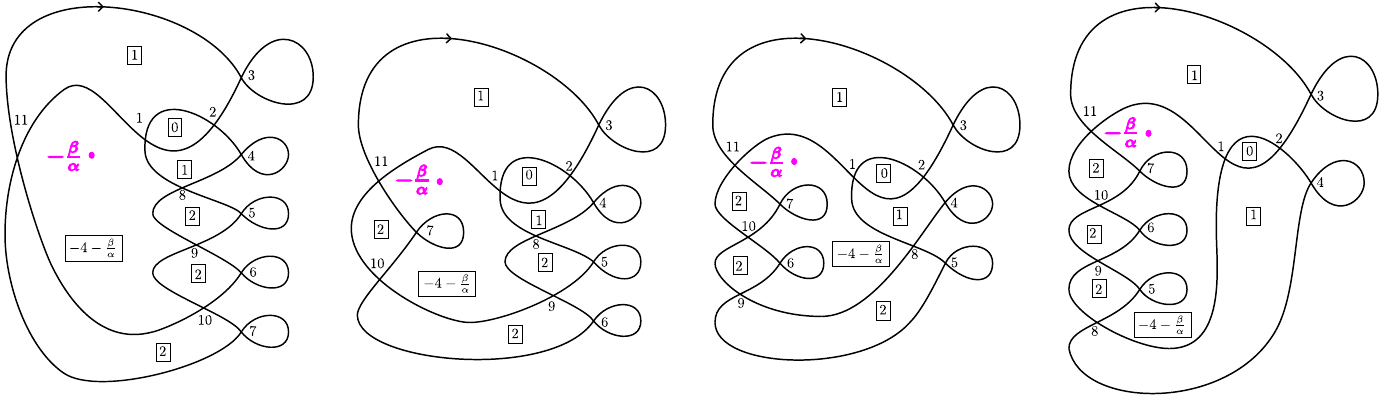}
    \caption{Legendrian representatives (left to right) $\Lambda(0,3)$, $\Lambda(1,2)$, $\Lambda(2,1)$, and $\Lambda(3,0)$ of prime $-5$-twist knot in $L(\alpha, \beta)$. The boxed numbers indicate defects of the respective regions.}
    \label{fig: 4 twist prime knots}
\end{figure}We begin by computing the gradings of the generators. For $\Lambda(0,3)$ the gradings are
\begin{align}
    |a_{3,\dots,7}|&= 1, &|b_{3,\dots,7}|&=2\mu -2,\\
    |a_{8,\dots,11}|&= 2\mu -1, &|b_{8,\dots,11}| &= 0,\\
    |a^0_1| &= -3-2\mu\frac{\beta}{\alpha}, &|b^0_1| &= 2 + 2\mu\left(1 + \frac{\beta}{\alpha}\right).\\
    |a^0_2| &= 3+2\mu \frac{\beta}{\alpha}, &|b^0_2| &= -4 + 2\mu\left(1 - \frac{\beta}{\alpha}\right).
\end{align}
We use superscript equal to $z$ on some of the generators indicate the Legendrian representative. The gradings for the knots $\Lambda(1,2)$ and $\Lambda(2,1)$ agree with those of $\Lambda(0,3)$ for $|a_{3,\dots,7}|, |b_{3,\dots,7}|, |a_{8,\dots,11}|$, and $|b_{8,\dots,11}|$. For the rest, the gradings are as follows. For $\Lambda(1,2)$ the gradings are
\begin{align}
    |a^1_1| &= -1-2\mu\frac{\beta}{\alpha}, &|b^1_1| &= 0 + 2\mu\left(1 + \frac{\beta}{\alpha}\right),\\
    |a^1_2| &= 1+2\mu \frac{\beta}{\alpha}, &|b^1_2| &= -2 + 2\mu\left(1 - \frac{\beta}{\alpha}\right);
\end{align}
for $\Lambda(2,1)$ the gradings are
\begin{align}
    |a^2_1| &= 1-2\mu\frac{\beta}{\alpha}, &|b^3_1| &= -2 + 2\mu\left(1 + \frac{\beta}{\alpha}\right),\\
    |a^2_2| &= -1+2\mu \frac{\beta}{\alpha}, &|b^3_2| &= 0 + 2\mu\left(1 - \frac{\beta}{\alpha}\right);
\end{align}
and for $\Lambda(3,0)$ the gradings are
\begin{align}
    |a^2_1| &= 3-2\mu\frac{\beta}{\alpha}, &|b^3_1| &= -4 + 2\mu\left(1 + \frac{\beta}{\alpha}\right).\\
    |a^2_2| &= -3+2\mu \frac{\beta}{\alpha}, &|b^3_2| &= 2 + 2\mu\left(1 - \frac{\beta}{\alpha}\right).
\end{align}
The next step is to find augmentations of each DGA, i.e., graded algebra maps $\epsilon: \mathcal{A} \to \Z_2$ that vanish on the image of $\partial$ and are supported in grading $0$. For this it is enough to check conditions for $\epsilon$ on the image of $\partial$ on grading $1$ generators. The differentials for $\Lambda(0,3)$ in degree $1$ are
\begin{align}
    \partial a_3 & = 1 + b_{11}, & \partial a_4 & = 1 + b_8,\\
    \partial a_5 & = 1 + b_8b_{9},
    & \partial a_6 & = 1 + b_{9}b_{10},\\
    \partial a_7 & = 1 + b_{10}b_{11},&&\\ \p b_{3, \dots, 7} &= 0,\\
\end{align}
The differentials for $\Lambda(1,2)$, $\Lambda(2,1)$, and $\Lambda(3,0)$ are similar up to rotation of the generators within a term. This means in all four cases, there is a unique augmentation $\epsilon$ that sends $b_8, \dots, b_{11}$ to $1$ and all other generators to $0$.

Now we can linearize the differentials by conjugating by $\phi^\epsilon: \A \to \A$,  $\psi^\epsilon(x) = x + \epsilon(x)$, and taking the linear terms. This gives us, for each of $\Lambda(0,3)$, $\Lambda(1,2)$, $\Lambda(2,1)$, and $\Lambda(3,0)$,
\begin{align}
\p^\epsilon a_{1,2} & 
= 0, & \p^\epsilon b_{1,2} &
= 0, 
 &\p b_{8, \dots, 11} &= 0, &\p b_{3, \dots, 7}& = 0,&&\\
    \p^\epsilon a_3 & = b_{11},  &\p^\epsilon a_4 &= b_8,
    &\p^\epsilon a_5 &= b_8 + b_{9},&&&&\\
    \p^\epsilon a_6 &=  b_{9} + b_{10}, &\p^\epsilon a_7 &= b_{10} + b_{11},&&&&&&\\
   \p^\epsilon a_8 & = b_4 +  b_5, &    \partial a_9 & = b_5 + b_6,&
    \partial a_{10} & = b_6 + b_7, &
    \partial a_{11} & = b_7 + b_3 ,&&
\end{align}
Thus, the Poincar\`e polynomials of the linearized LCHs are
\begin{align}
    p(\Lambda(0,3)) & = t + t^{2\mu -2} + t^{3+2\mu \frac{\beta}{\alpha}} + t^{2 + 2\mu\left(1 + \frac{\beta}{\alpha}\right)} +  t^{-3-2\mu\frac{\beta}{\alpha}} + t^{-4 + 2\mu\left(1 - \frac{\beta}{\alpha}\right)},
\\
    p(\Lambda(1,2)) &= t + t^{2\mu -2} + t^{1+2\mu \frac{\beta}{\alpha}} + t^{2\mu\left(1 + \frac{\beta}{\alpha}\right)}  + t^{-1-2\mu\frac{\beta}{\alpha}} + t^{-2 + 2\mu\left(1 - \frac{\beta}{\alpha}\right)},
\\
    p(\Lambda(2,1)) &= t + t^{2\mu -2} + t^{1-2\mu \frac{\beta}{\alpha}} + t^{2\mu\left(1 - \frac{\beta}{\alpha}\right)}  + t^{-1+2\mu\frac{\beta}{\alpha}} + t^{-2 + 2\mu\left(1 + \frac{\beta}{\alpha}\right)},
 \\
    p(\Lambda(3,0)) &= t + t^{2\mu -2} + t^{3-2\mu \frac{\beta}{\alpha}} + t^{2 + 2\mu\left(1 - \frac{\beta}{\alpha}\right)}  + t^{-3+2\mu\frac{\beta}{\alpha}} + t^{-4 + 2\mu\left(1 + \frac{\beta}{\alpha}\right)}.   
\end{align}
As $\mu \neq 0$ and $\beta/\alpha \neq 0$, we get four distinct polynomials, and therefore, all four of $\Lambda(0,3)$, $\Lambda(1,2)$, $\Lambda(2,1)$, and $\Lambda(3,0)$ are non-isotopic Legendrians.

These computations work more generally for any $(-n)$-twist knot. The gradings differ slightly for odd and even twists.

Let us first consider an even number of twists, specifically the $(-n)$-twist knot, where $n = 2k, k\geq 2$.

We always get $2n+1$ generators. Let us label them as in Figure~\ref{fig:n twist knot}, namely, $1$ and $2$ for the crossings at the clasp, $3, \dots, n+2$, for crossings at the tear drops, and $n+3, \dots, 2n+1$, for the crossings from the twists. 
Then we get gradings for $\Lambda(z, s)$, $0 \leq  z \leq n-2$, $z+s = n-2$, 
\begin{align}
    |a_{3, \dots, n+2}| & = 1, &|b_{3, \dots, n+2}| &= 2\mu-2,\\
    |a_{n+3, \dots, 2n+1}| & = 2\mu - 1, &|b_{n+3, \dots, 2n+1}|&=0,\\
    |a_1^z| & = s-z+3, &|b_1^z| &= z-s+2 + 2\mu,\\
    |a_2^z|& = z-s-1 + 2\mu\left(1-\frac{\beta}{\alpha}\right), &|b_2^z| &= s-z + 2\mu\frac{\beta}{\alpha}.  
\end{align}
In each case, there exists a unique augmentation that sends the generators $b_{n+3}, \dots, b_{2n+1}$ to $1$ and the rest of the generators to zero. Once we linearize the differential using this augmentation, the corresponding Poincar\'{e} polynomial is given by
\begin{align}\label{eqn:poincare polynomial prime twist leg}
      P(\Lambda(z,s)) = t^1 + t^{2\mu -2} + t^{|a^z_1|} + t^{|b^z_1|} +  t^{|a^z_2|} + t^{|b^z_2|}.
\end{align}
Clearly, the Poincar\'e polynomials are all distinct for distinct $z$ as long as $\beta/\alpha$ are nonzero. Hence, each $\Lambda(z, s)$ is a distinct Legendrian $(-n)$-twist knot in $L(\alpha, \beta)$, for $n \geq 4$ even.
\begin{figure}[h]
    \centering
    \includegraphics[width=\linewidth]{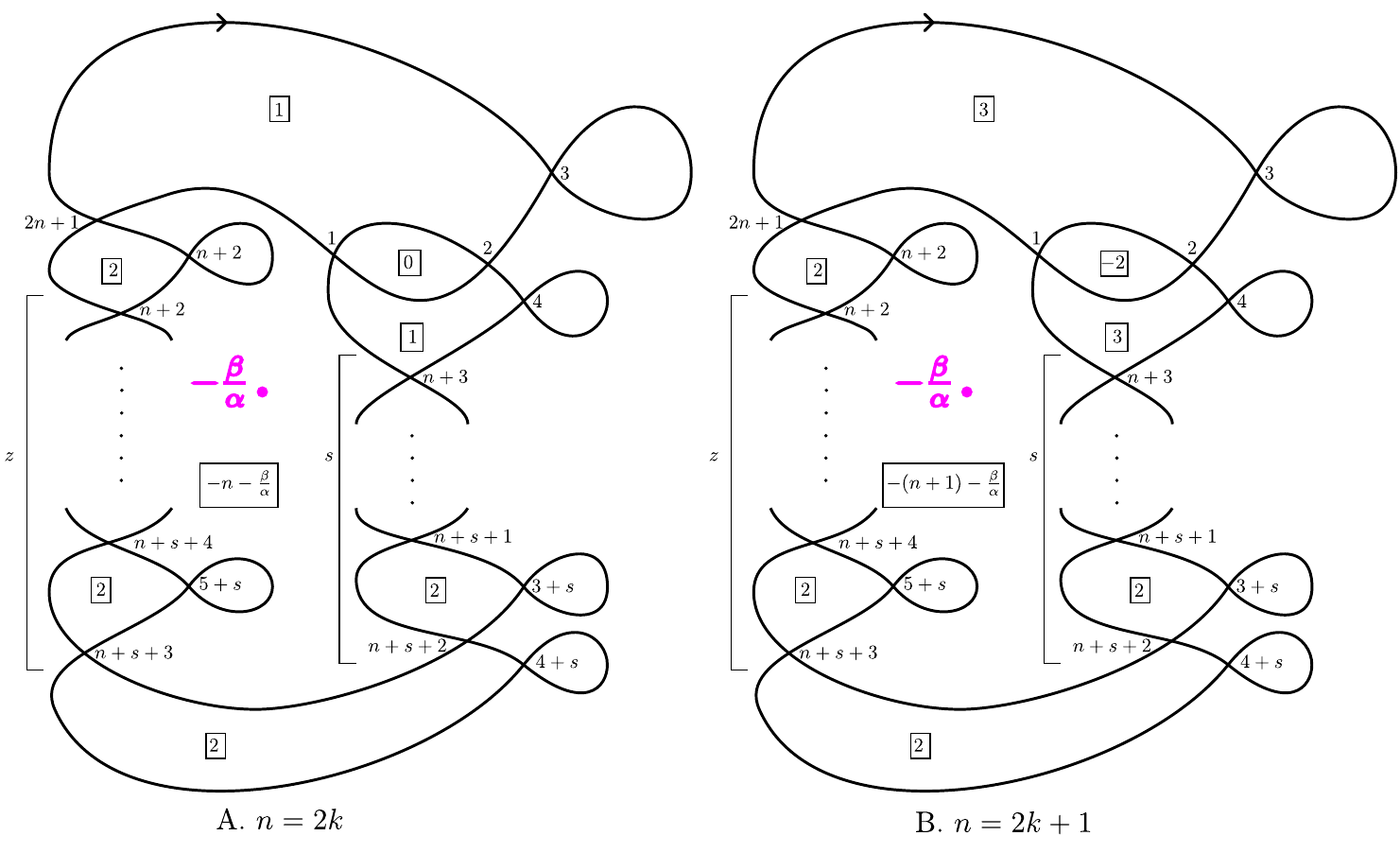}
    \caption{Legendrians $\Lambda(z, s)$ for $z+s = n-2$. The boxed numbers indicate defects of the respective regions.}
    \label{fig:n twist knot}
\end{figure}

Now, we consider an odd number of twists, specifically the $(-n)$-twist knot, where $n = 2k+1, k \geq 1$. 
We again get $2n+1$ generators, labelled as we did in the $n$ even case, as in Figure~\ref{fig:n twist knot}. 
For $\Lambda(z,s)$, $0\leq z \leq n-2$, $z+s = n-2$, the gradings are
\begin{align}
|a_{3, \dots, n+2}| & = 1, &|b_{3, \dots, n+2}| &= 2\mu-2,\\
    |a_{n+3, \dots, 2n+1}| & = 2\mu - 1, &|b_{n+3, \dots, 2n+1}|&=0,\\
    &|a_1^z| = z-s-2\mu\frac{\beta}{\alpha}, &|b_1^z| &= s-z-1 + 2\mu\left(1 + \frac{\beta}{\alpha}\right).\\
    &|a_2^z| = s-z+2\mu \frac{\beta}{\alpha}, &|b_2^z| & = z-s-1 + 2\mu\left(1 - \frac{\beta}{\alpha}\right).
\end{align}
Again, in each case, there exists a unique augmentation that sends the generators $b_{n+3}, \dots, b_{2n+1}$ to $1$ and the remaining generators to $0$. Once we linearize the differential using this augmentation, the corresponding Poincar\'{e} polynomials are given by
\begin{align}
      P(\Lambda(z,s)) = t^1 + t^{2\mu -2} + t^{|a^z_1|} + t^{|b^z_1|} +  t^{|a^z_2|} + t^{|b^z_2|},
\end{align}
which are distinct for distinct $z$ as long as $\beta/\alpha$ is nonzero. Hence, each $\Lambda(z, s)$ is a distinct Legendrian $(-n)$-twist knot in $L(\alpha, \beta)$ for $n$ odd.

\subsection{Non-simple non-prime twist knots}\label{sec: non prime twist knots}
In this section, we expand on the pair of torsion knots from \cite{Licata-Sabloff_2013}, giving a proof of Theorem~\ref{thm:licata-sabloff legendrian examples}. We describe their construction in a slightly different way. Repeat the same construction as in Section~\ref{sec: prime twist knots} but with a Legendrian $\Lambda_0$ such that it intersects the base $\Lambda_0 \cap S^2 = \{p_2\}$ and has Weinstein torus such that $\nu \Lambda_0 \cap S^2 \subset R_2$. Then $\Gamma$ can be lifted to a Legendrian isotopic to the $(-n)$-twist knot.

In the notation of \cite{Licata-Sabloff_2013}, this knot is represented by the labelled Lagrangian diagram that has ``defects'' given as in Figure~\ref{fig:n twist licata sabloff examples}. This gives the same Legendrians as \cite{Licata-Sabloff_2013} for $n = 5$. These are not prime knots, as they can be written as
\begin{align}
    \Lambda'(z, s) = E(z, s) \# F, \quad z+s = n-2,
\end{align}
where $E(z, s)$ is a $(-n)$-twist Legendrian in $(S^3, \xi_\std)$ and $F$ is a Legendrian isotopic to a regular fiber in $L(\alpha, \beta)$ with Lagrangian projection as in Figure~\ref{fig:prime decomposition}.

\newpage
\begin{figure}[h]
    \centering
    \includegraphics[width=\linewidth]{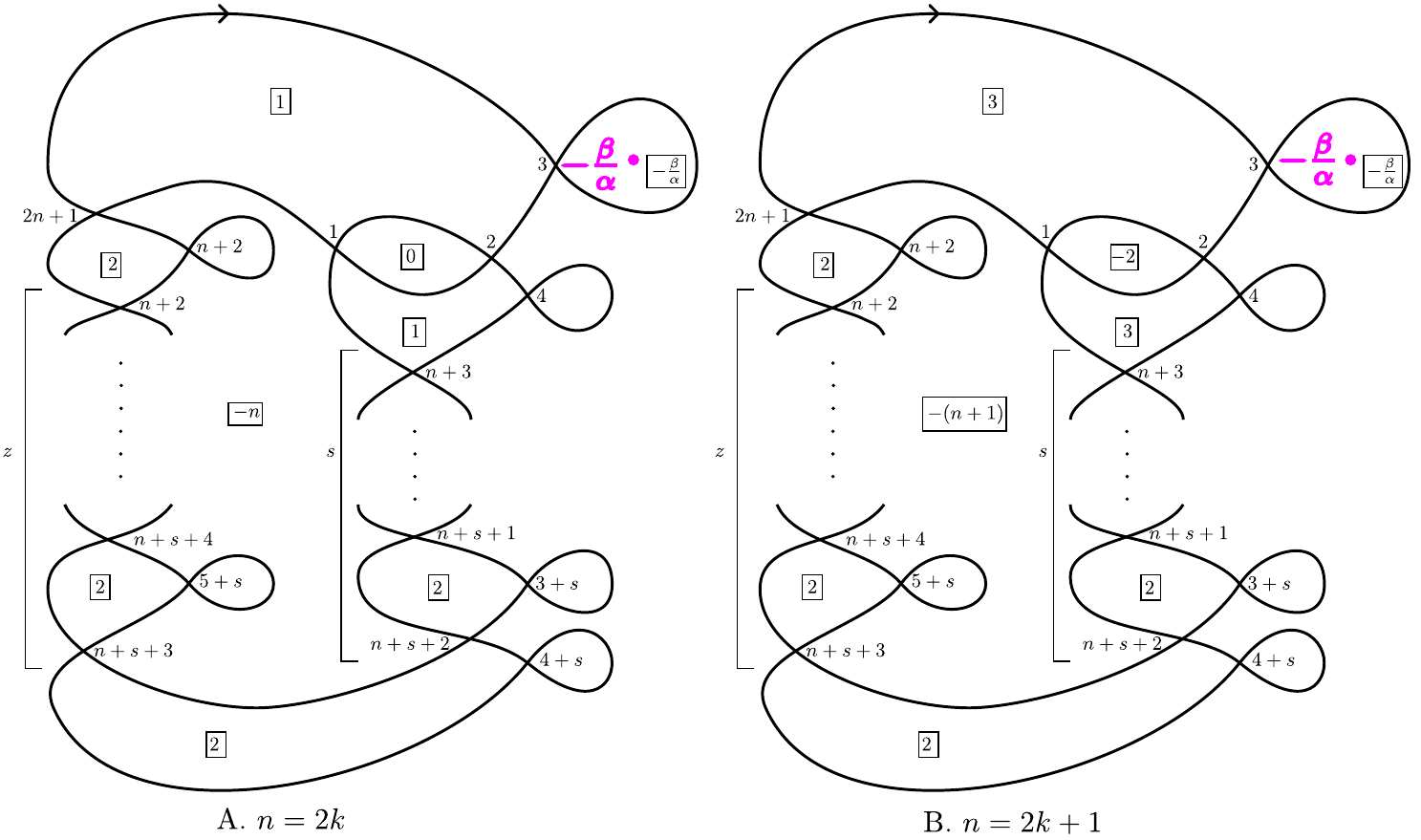}
    \caption{Legendrians $\Lambda'(z,s)$, for $z+s = n-2$, in $L(\alpha, \beta)$. The boxed numbers indicate defects of the respective regions.}
    \label{fig:n twist licata sabloff examples}
\end{figure}

\begin{figure}[h]
    \centering
    \includegraphics[width=0.7\linewidth]{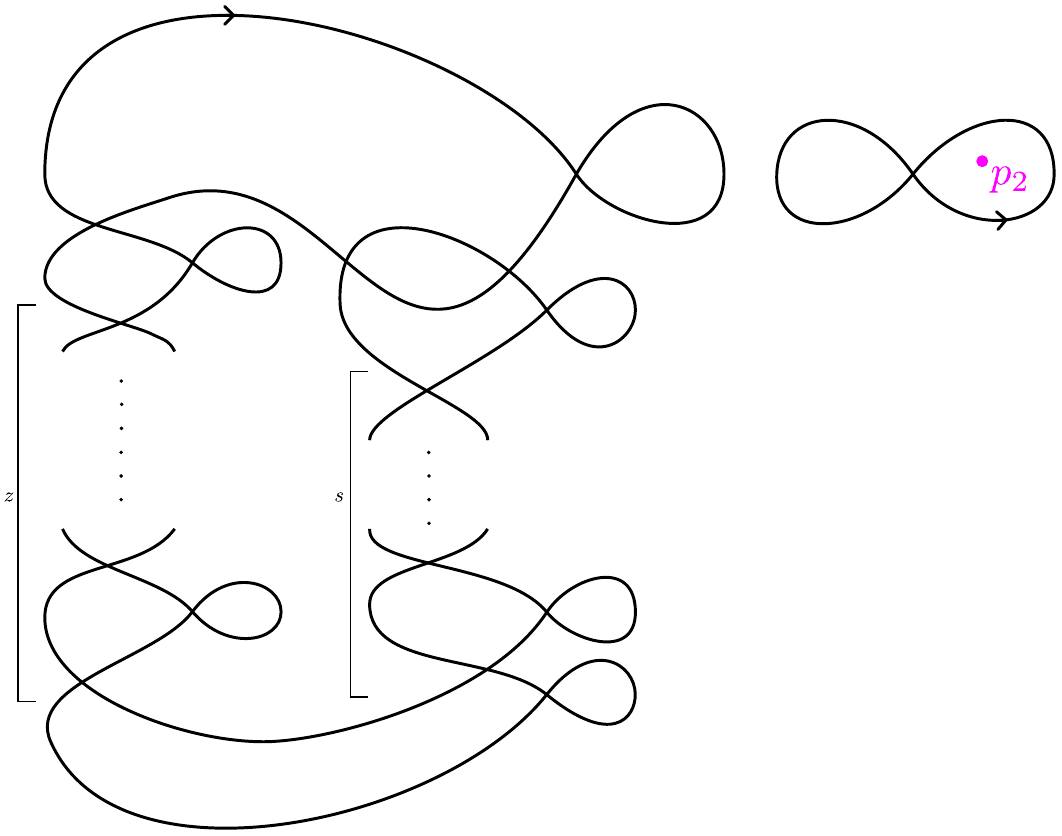}
    \caption{Prime decomposition of $\Lambda'(z,s)$: $E(z,s)$ on the left and $F$ on the right.}
    \label{fig:prime decomposition}
\end{figure}

The classical rational invariants of these knots are (see \cite{kegel_2018})
\begin{align}
    \tb_\Q(\Lambda'(z,s)) &= 1 +\frac{\alpha}{\beta} &\text{ for } k \text{ even, }\\
    \tb_\Q(\Lambda'(z,s)) &= -3 +\frac{\alpha}{\beta} &\text{ for } k \text{ odd, and }\\
    \rot_\Q(\Lambda'(z,s)) & = \frac{\alpha}{\beta}\rot(\Lambda_0) & \text{ for all }k.
\end{align}
To see this, note that the linking number between the surgery Legendrian and the Legendrian lift of $\Gamma$ in $S^3$ has linking number $1$.

As before, the grading is only well-defined in $\Z$ modulo 
\begin{equation}2r(S) + 2\mu n(S),\end{equation}
for $S$ a formal capping surface of $\Gamma$.
Note that for all the $n$-twist knots $\Lambda'(z,s)$ we described, we can choose a formal capping surface $S$ such that the total rotation $r(S) = 0$, but the defect is $n(S) = \frac{\beta}{\alpha}$. Hence, the gradings are defined modulo $2\mu \frac{\beta}{\alpha}$.

Label the generators as in Figure~\ref{fig:n twist licata sabloff examples}.
Then, for $n=2k \geq 4$, we get that the gradings for $\Lambda'(z,s)$, $0\leq z \leq n-2$, $z+s = n-2$, 
\begin{align}
    |a_{3, \dots, n+2}| & = 1, &|b_{3, \dots, n+2}| &= 2\mu-2,\\
    |a_{n+3, \dots, 2n+1}| & = 2\mu - 1, &|b_{n+3, \dots, 2n+1}|&=0,\\
    |a_1^l| & = s-z+3, &|b_1^l| &= z-s-4 + 2\mu,\\
    |a_2^l|& = z-s-1 + 2\mu, &|b_2^l| &= s-z .  
\end{align}

For $n=2k+1 \geq 3$, $\Lambda'(z,s)$, $0\leq z \leq n-2$, $z+s = n-2$, the gradings are
\begin{align}
|a_{3, \dots, n+2}| & = 1, &|b_{3, \dots, n+2}| &= 2\mu-2,\\
    |a_{n+3, \dots, 2n+1}| & = 2\mu - 1, &|b_{n+3, \dots, 2n+1}|&=0,\\
    &|a_1| = z-s, &|b_1| &= s-z+ 2\mu -1.\\
    &|a_2| = s-z, &|b_2| & = z-s + 2\mu -1.
\end{align}
Again, in each case, there exists a unique augmentation that sends the generators $b_{n+3}, \dots, b_{2n+1}$ to $1$ and the remaining generators to $0$. Once we linearize the differential using this augmentation, the corresponding Poincar\'{e} polynomials are given by
\begin{align}
      P(\Lambda(l, n-l)) = t^1 + t^{2\mu -2} + t^{|a^z_1|} + t^{|b^z_1|} +  t^{|a^z_2|} + t^{|b^z_2|},
\end{align}
which are distinct for distinct $z \leq \lceil (n + 1)/2 \rceil$. Hence, each $\Lambda(z, s)$, $z \leq \lceil (n+1)/2 \rceil$, is a distinct Legendrian.

\begin{rem}
    The Poincar\'e polynomials for the $n=5$ case look different from \cite{Licata-Sabloff_2013} due to a small computation error, namely, the Lagrangian projections in the cited paper are missing one lobe each. This error does not take away from any results in that paper or even the given example. It is still true that the two torsion knots presented are not Legendrian isotopic, and this can be concluded by noticing that in one case, there is a generator in grading $3$ that is not present in the LCH of the other. Additionally, the linearized LCH of these knots has a ``fundamental class'' in degree $1$ just as in $\R^3$.
\end{rem}

\section{Non-Simple Legendrian and Transverse Cables}
In this section, we prove Theorem~\ref{thm:cable of torus knot} and Theorem~\ref{thm:cable of twist knot}.
We first need to recall the definition of the cable of a knot and the contact width. 

A {\bf Legendrian $(p,q)$--cable} of a Legendrian knot $K$ is a Legendrian knot $K_{p,q}$ obtained by realizing a $(p,q)$--curve on the convex boundary of a standard neighbourhood of $K$ and pushing it slightly into the interior so it becomes Legendrian. When $p$ and $q$ are not co-prime, the $(p,q)$--curve is a link with $\gcd(p,q)$ components; in what follows, we assume $\gcd(p,q)=1$ so the cable is a knot. One can construct a transverse cable by taking the transverse push-off of the Legendrian cable.

The {\bf contact width $\omega(K)$} is defined as follows in \cite{EH02}: First, an embedding $\phi: S^1 \times D^2 \hookrightarrow S^3$ is said to {\bf represent} $K$ if the core curve of $\phi(S^1 \times D^2)$ is isotopic to $K$. Next, make a (somewhat nonstandard) oriented identification of $\partial(S^1 \times D^2) \simeq \R^2/\Z^2$, where the meridian has slope $0$ and the longitude (well-defined since $K$ is inside $S^3$) has slope $\infty$. Call this coordinate system $C_K$. Define
\begin{align}
    \omega(K) = \sup \frac{1}{\mathrm{slope}(\Gamma_{\partial(S^1 \times D^2)})}
\end{align}
where the supremum us taken over all embeddings $S^1 \times D^2 \hookrightarrow S^3$ representing $K$ with $\partial(S^1 \times D^2)$ convex and slope is with respect to the coordinate system $C_K$.

\begin{proof}[Proof of Theorem~\ref{thm:cable of torus knot}]
    
    Let $T_1$ be a torus knot in $\mathbb{S}^3$ and $T_2$ be a torus knot on the Heegaard torus in a Lens space $L(p,q)$. For some such torus knots, the connected sum $T_1 \# T_2$ is Legendrian and transversely non-simple. 
    
    For example, consider $T_1=T_{-5,3}$ in $\mathbb{S}^3$ and $T_2=T_{-5,2}$ in any $L(\alpha, \beta)$. Let $\Lambda_1^\pm$ be two maximum $\tb$ Legendrian representatives of $T_1$ with rotation numbers $\pm 2$, respectively, see \cite{EH02}. Let $\Lambda_2$ and $\Lambda'_2$ be two max $\tb$ representatives of $T_2$ with rotation numbers $5+2\times(\frac{1+q}{p})-4$ and $5+2\times(\frac{1+q}{p})-8$, respectively, see \cites{Ona, Zha}. 
    By Theorem 3.4 in \cite{Etnyre-Honda_2003_connect_sums}, the two connect sums 
    \begin{align}
        \Lambda_1^- \# \Lambda_2 \text{ and } \Lambda_1^+ \# \Lambda'_2,
    \end{align}
    are not Legendrian isotopic, but by using the formulas in Theorem~\ref{thm: tb under connect sum}, one can check that these have the same $\tb_\Q$ and $\rot_\Q$.

    This construction can be repeated with any other $T_i$, $i=1,2$, from classifications in \cites{Ona, Zha, EH02}, such that the $\tb_\Q$ and $\rot_\Q$ of the connected sums match, but the prime decompositions consist of non-isotopic Legendrians. The uniqueness of prime decomposition in Theorem 3.4 in \cite{Etnyre-Honda_2003_connect_sums} will imply that these connect sums are not Legendrian isotopic.

    Let $C_{p,q}(T_1 \# T_2)$ be a large positive cable of a Legendrian non-simple knot \( T_1 \# T_2 \) as constructed above. When the cabling slope of a non-simple Legendrian knot is \( \frac{q}{p} > \lceil\omega(K) \rceil\), \cite[Theorem 1.1]{CEM} implies the cable knot is again non-simple.

Consider $L_1$ and $L_2$ to be two max $\tb$ representatives of $C_{p,q}(T_1\# T_2)$. Let $S^-(L_i)$ denote negative stabilizations of $L_i$, $i=1,2$. These stabilizations $S^-(L_i)$ are not Legendrian isotopic by \cite[Theorem 3.4]{Etnyre-Honda_2003_connect_sums} as stabilizations can be realised as connected sums.
Consider transverse pushoffs $T(S^-(L_i))$, $i=1,2$, of these non-isotopic Legendrians. The following Theorem~\ref{thm:classification of transverse equiv to leg upto stabilization} from \cites{EFM, EH03} implies $T(S^-(L_i))$, $i=1,2$ are not transverse isotopic.
\begin{thm}\cites{EFM, EH03}\label{thm:classification of transverse equiv to leg upto stabilization}
 The classification of transverse knots up to transverse isotopy is equivalent to the classification of Legendrian knots up to negative stabilization and Legendrian isotopy.
\end{thm}
\end{proof}

\begin{proof}[Proof of Theorem~\ref{thm:cable of twist knot}]
    Consider Legendrian non-isotopic $\Lambda(z,s)$, $0\leq z \leq n$, $z+s = n-2$, $n\geq 3$, from Theorem~\ref{thm:prime legendrian examples}. Sufficiently large positive cables $C_{p,q}(\Lambda(z,s))$, namely when the cabling slope \( \frac{q}{p} > \lceil\omega(\Lambda(z,s)) \rceil\), $0\leq z \leq n$, $z+s = n-2$, are also not Legendrian isotopic.

Similarly, sufficiently large positive cables of the non-isotopic knots $\Lambda'(z,s)$, $0\leq z \leq n$, $z+s = n-2$, , $n\geq 3$ give another class of non-isotopic Legendrians.
\end{proof}

\bibliographystyle{amsplain}
\bibliography{references.bib}

@article{Gab,
   title={Tabulation of Prime Knots in Lens Spaces},
   volume={14},
   ISSN={1660-5454},
   url={http://dx.doi.org/10.1007/s00009-016-0814-5},
   DOI={10.1007/s00009-016-0814-5},
   number={2},
   journal={Mediterranean Journal of Mathematics},
   publisher={Springer Science and Business Media LLC},
   author={Gabrovšek, Boštjan},
   year={2017},
   month=mar }

@incollection {E,
    AUTHOR = {Etnyre, John B.},
     TITLE = {Legendrian and transversal knots},
 BOOKTITLE = {Handbook of knot theory},
     PAGES = {105--185},
 PUBLISHER = {Elsevier B. V., Amsterdam},
      YEAR = {2005},
   MRCLASS = {57R17 (53D35 57M25 57M27)},
  MRNUMBER = {2179261},
MRREVIEWER = {Lenhard L. Ng},
       DOI = {10.1016/B978-044451452-3/50004-6},
       URL = {https://doi.org/10.1016/B978-044451452-3/50004-6},
}

@misc{BE,
      title={Rational linking and contact geometry}, 
      author={Kenneth L. Baker and John B. Etnyre},
      year={2014},
      eprint={0901.0380},
      archivePrefix={arXiv},
      primaryClass={math.SG},
      url={https://arxiv.org/abs/0901.0380}, 
}

@article {ekholm-etnyre-sullivan_2005,
    AUTHOR = {Ekholm, Tobias and Etnyre, John and Sullivan, Michael},
     TITLE = {The contact homology of {L}egendrian submanifolds in {${\mathbb
              R}^{2n+1}$}},
   JOURNAL = {J. Differential Geom.},
  FJOURNAL = {Journal of Differential Geometry},
    VOLUME = {71},
      YEAR = {2005},
    NUMBER = {2},
     PAGES = {177--305},
      ISSN = {0022-040X,1945-743X},
   MRCLASS = {53D35 (57R17)},
  MRNUMBER = {2197142},
MRREVIEWER = {Joshua\ M.\ Sabloff},
       URL = {http://projecteuclid.org/euclid.jdg/1143651770},
}

@misc{ekholm-ng,
      title={Legendrian contact homology in the boundary of a subcritical Weinstein 4-manifold}, 
      author={Tobias Ekholm and Lenhard Ng},
      year={2014},
      eprint={1307.8436},
      archivePrefix={arXiv},
      primaryClass={math.SG},
      url={https://arxiv.org/abs/1307.8436}, 
}

@article{etnyre-ng-vertesi,
   title={Legendrian and transverse twist knots},
   volume={15},
   ISSN={1435-9863},
   url={http://dx.doi.org/10.4171/JEMS/383},
   DOI={10.4171/jems/383},
   number={3},
   journal={Journal of the European Mathematical Society},
   publisher={European Mathematical Society - EMS - Publishing House GmbH},
   author={Etnyre, John B. and Ng, Lenhard L. and Vértesi, Vera},
   year={2013},
   month=mar, pages={969–995} }

@article{Zha,
  title={Legendrian negative torus knots in universally tight lens spaces},
  author={Zhang, Han},
  journal={arXiv preprint arXiv:2302.04199},
  year={2023}
}

@misc{EH02,
      title={Cabling and transverse simplicity}, 
      author={John B. Etnyre and Ko Honda},
      year={2007},
      eprint={math/0306330},
      archivePrefix={arXiv},
      primaryClass={math.SG},
      url={https://arxiv.org/abs/math/0306330}, 
}

@misc{CEM,
      title={Cabling Legendrian and transverse knots}, 
      author={Apratim Chakraborty and John B. Etnyre and Hyunki Min},
      year={2021},
      eprint={2012.12148},
      archivePrefix={arXiv},
      primaryClass={math.GT},
      url={https://arxiv.org/abs/2012.12148}, 
}

@article{EH03,
  title={Knots and contact geometry I: torus knots and the figure eight knot},
  author={Etnyre, John B and Honda, Ko},
  year={2001}
}

@article{EFM,
  title={Chekanov--Eliashberg invariants and transverse approximations of Legendrian knots},
  author={Epstein, Judith and Fuchs, Dmitry and Meyer, Maike},
  journal={Pacific Journal of Mathematics},
  volume={201},
  number={1},
  pages={89--106},
  year={2001},
  publisher={Mathematical Sciences Publishers}
}

@article{Ona,
  title={Legendrian knots in Lens spaces},
  author={Onaran, Sinem},
  year={2011},
  publisher={Mathematisches Forschungsinstitut Oberwolfach}
}

@misc{GO,
      title={Legendrian rational unknots in lens spaces}, 
      author={Hansjörg Geiges and Sinem Onaran},
      year={2013},
      eprint={1302.3792},
      archivePrefix={arXiv},
      primaryClass={math.SG},
      url={https://arxiv.org/abs/1302.3792}, 
}

@article {Sabloff_2003,
    AUTHOR = {Sabloff, Joshua M.},
     TITLE = {Invariants of {L}egendrian knots in circle bundles},
   JOURNAL = {Commun. Contemp. Math.},
  FJOURNAL = {Communications in Contemporary Mathematics},
    VOLUME = {5},
      YEAR = {2003},
    NUMBER = {4},
     PAGES = {569--627},
      ISSN = {0219-1997,1793-6683},
   MRCLASS = {57R17 (53D35 57M27)},
  MRNUMBER = {2003211},
MRREVIEWER = {Vladimir\ V.\ Tchernov},
       DOI = {10.1142/S0219199703001075},
       URL = {https://doi.org/10.1142/S0219199703001075},
}

@article {Licata_2011,
    AUTHOR = {Licata, Joan E.},
     TITLE = {Invariants for {L}egendrian knots in lens spaces},
   JOURNAL = {Commun. Contemp. Math.},
  FJOURNAL = {Communications in Contemporary Mathematics},
    VOLUME = {13},
      YEAR = {2011},
    NUMBER = {1},
     PAGES = {91--121},
      ISSN = {0219-1997,1793-6683},
   MRCLASS = {57R17 (57M27)},
  MRNUMBER = {2772580},
MRREVIEWER = {David\ Shea\ Vela-Vick},
       DOI = {10.1142/S0219199711004178},
       URL = {https://doi.org/10.1142/S0219199711004178},
}

@article {Licata-Sabloff_2013,
    AUTHOR = {Licata, Joan E. and Sabloff, Joshua M.},
     TITLE = {Legendrian contact homology in {S}eifert fibered spaces},
   JOURNAL = {Quantum Topol.},
  FJOURNAL = {Quantum Topology},
    VOLUME = {4},
      YEAR = {2013},
    NUMBER = {3},
     PAGES = {265--301},
      ISSN = {1663-487X,1664-073X},
   MRCLASS = {53D42 (57R17 57R18)},
  MRNUMBER = {3073564},
       DOI = {10.4171/QT/40},
       URL = {https://doi.org/10.4171/QT/40},
}

@article {Licata-Sabloff_2012,
    AUTHOR = {Licata, Joan E. and Sabloff, Joshua M.},
     TITLE = {Rational {S}eifert surfaces in {S}eifert fibered spaces},
   JOURNAL = {Pacific J. Math.},
  FJOURNAL = {Pacific Journal of Mathematics},
    VOLUME = {258},
      YEAR = {2012},
    NUMBER = {1},
     PAGES = {199--221},
      ISSN = {0030-8730,1945-5844},
   MRCLASS = {57M27},
  MRNUMBER = {2972483},
MRREVIEWER = {Torsten\ Asselmeyer-Maluga},
       DOI = {10.2140/pjm.2012.258.199},
       URL = {https://doi.org/10.2140/pjm.2012.258.199},
}

@article {EF,
    AUTHOR = {Eliashberg, Yakov and Fraser, Maia},
     TITLE = {Topologically trivial {L}egendrian knots},
   JOURNAL = {J. Symplectic Geom.},
  FJOURNAL = {The Journal of Symplectic Geometry},
    VOLUME = {7},
      YEAR = {2009},
    NUMBER = {2},
     PAGES = {77--127},
      ISSN = {1527-5256},
   MRCLASS = {57R17 (57M25 57N37)},
  MRNUMBER = {2496415},
MRREVIEWER = {Paolo Lisca},
       URL = {http://projecteuclid.org/euclid.jsg/1239974381},
}

@article {C1,
    AUTHOR = {Chekanov, Yuri},
     TITLE = {Differential algebra of {L}egendrian links},
   JOURNAL = {Invent. Math.},
  FJOURNAL = {Inventiones Mathematicae},
    VOLUME = {150},
      YEAR = {2002},
    NUMBER = {3},
     PAGES = {441--483},
      ISSN = {0020-9910},
   MRCLASS = {53D35 (57M27 57R17)},
  MRNUMBER = {1946550},
MRREVIEWER = {John B. Etnyre},
       DOI = {10.1007/s002220200212},
       URL = {https://doi.org/10.1007/s002220200212},
}

@inproceedings {C2,
    AUTHOR = {Chekanov, Yu. V.},
     TITLE = {Invariants of {L}egendrian knots},
 BOOKTITLE = {Proceedings of the {I}nternational {C}ongress of
              {M}athematicians, {V}ol. {II} ({B}eijing, 2002)},
     PAGES = {385--394},
 PUBLISHER = {Higher Ed. Press, Beijing},
      YEAR = {2002},
   MRCLASS = {57M27 (57R17)},
  MRNUMBER = {1957049},
}

@article {colin_1997,
    AUTHOR = {Colin, Vincent},
     TITLE = {Chirurgies d'indice un et isotopies de sph\`{e}res dans les
              vari\'{e}t\'{e}s de contact tendues},
   JOURNAL = {C. R. Acad. Sci. Paris S\'{e}r. I Math.},
  FJOURNAL = {Comptes Rendus de l'Acad\'{e}mie des Sciences. S\'{e}rie I.
              Math\'{e}matique},
    VOLUME = {324},
      YEAR = {1997},
    NUMBER = {6},
     PAGES = {659--663},
      ISSN = {0764-4442},
   MRCLASS = {57R15 (57R67 58F05)},
  MRNUMBER = {1447038},
       DOI = {10.1016/S0764-4442(97)86985-6},
       URL = {https://doi.org/10.1016/S0764-4442(97)86985-6},
}

@article {baker-grigsby_2009,
    AUTHOR = {Baker, Kenneth L. and Grigsby, J. Elisenda},
     TITLE = {Grid diagrams and {L}egendrian lens space links},
   JOURNAL = {J. Symplectic Geom.},
  FJOURNAL = {The Journal of Symplectic Geometry},
    VOLUME = {7},
      YEAR = {2009},
    NUMBER = {4},
     PAGES = {415--448},
      ISSN = {1527-5256,1540-2347},
   MRCLASS = {57M25 (57R17)},
  MRNUMBER = {2552000},
MRREVIEWER = {Quach thi C\^am V\^an},
       DOI = {10.4310/jsg.2009.v7.n4.a2},
       URL = {https://doi.org/10.4310/jsg.2009.v7.n4.a2},
}

@article {durst-kegel-klukas_2016,
    AUTHOR = {Durst, S. and Kegel, M. and Klukas, M.},
     TITLE = {Computing the {T}hurston-{B}ennequin invariant in open books},
   JOURNAL = {Acta Math. Hungar.},
  FJOURNAL = {Acta Mathematica Hungarica},
    VOLUME = {150},
      YEAR = {2016},
    NUMBER = {2},
     PAGES = {441--455},
      ISSN = {0236-5294,1588-2632},
   MRCLASS = {57R17 (57M27 57R65)},
  MRNUMBER = {3568102},
MRREVIEWER = {Burak\ Ozbagci},
       DOI = {10.1007/s10474-016-0648-4},
       URL = {https://doi.org/10.1007/s10474-016-0648-4},
}

@article {Etnyre-Honda_2003_connect_sums,
    AUTHOR = {Etnyre, John B. and Honda, Ko},
     TITLE = {On connected sums and {L}egendrian knots},
   JOURNAL = {Adv. Math.},
  FJOURNAL = {Advances in Mathematics},
    VOLUME = {179},
      YEAR = {2003},
    NUMBER = {1},
     PAGES = {59--74},
      ISSN = {0001-8708,1090-2082},
   MRCLASS = {57M25 (57M50 57R17)},
  MRNUMBER = {2004728},
MRREVIEWER = {Quach thi C\^{a}m V\^{a}n},
       DOI = {10.1016/S0001-8708(02)00027-0},
       URL = {https://doi.org/10.1016/S0001-8708(02)00027-0},
}

@article {honda_2000_classification_tight_1,
    AUTHOR = {Honda, Ko},
     TITLE = {On the classification of tight contact structures. {I}},
   JOURNAL = {Geom. Topol.},
  FJOURNAL = {Geometry and Topology},
    VOLUME = {4},
      YEAR = {2000},
     PAGES = {309--368},
      ISSN = {1465-3060,1364-0380},
   MRCLASS = {53D35 (57M50 57R17)},
  MRNUMBER = {1786111},
MRREVIEWER = {Hansj\"org\ Geiges},
       DOI = {10.2140/gt.2000.4.309},
       URL = {https://doi.org/10.2140/gt.2000.4.309},
}

@article {kamishima-tsuboi_1991,
    AUTHOR = {Kamishima, Yoshinobu and Tsuboi, Takashi},
     TITLE = {C{R}-structures on {S}eifert manifolds},
   JOURNAL = {Invent. Math.},
  FJOURNAL = {Inventiones Mathematicae},
    VOLUME = {104},
      YEAR = {1991},
    NUMBER = {1},
     PAGES = {149--163},
      ISSN = {0020-9910,1432-1297},
   MRCLASS = {53C15 (53C20 57N10)},
  MRNUMBER = {1094049},
MRREVIEWER = {William\ Goldman},
       DOI = {10.1007/BF01245069},
       URL = {https://doi.org/10.1007/BF01245069},
}

@article {lisca-matic_2004,
    AUTHOR = {Lisca, Paolo and Mati\'c, Gordana},
     TITLE = {Transverse contact structures on {S}eifert 3-manifolds},
   JOURNAL = {Algebr. Geom. Topol.},
  FJOURNAL = {Algebraic \& Geometric Topology},
    VOLUME = {4},
      YEAR = {2004},
     PAGES = {1125--1144},
      ISSN = {1472-2747,1472-2739},
   MRCLASS = {57R17 (53D35 57N10 57R30)},
  MRNUMBER = {2113899},
MRREVIEWER = {John\ B.\ Etnyre},
       DOI = {10.2140/agt.2004.4.1125},
       URL = {https://doi.org/10.2140/agt.2004.4.1125},
}

@article {epstein-fuchs-meyer,
    AUTHOR = {Epstein, Judith and Fuchs, Dmitry and Meyer, Maike},
     TITLE = {Chekanov-{E}liashberg invariants and transverse approximations
              of {L}egendrian knots},
   JOURNAL = {Pacific J. Math.},
  FJOURNAL = {Pacific Journal of Mathematics},
    VOLUME = {201},
      YEAR = {2001},
    NUMBER = {1},
     PAGES = {89--106},
      ISSN = {0030-8730,1945-5844},
   MRCLASS = {57M27 (53D10)},
  MRNUMBER = {1867893},
MRREVIEWER = {Serge\ L.\ Tabachnikov},
       DOI = {10.2140/pjm.2001.201.89},
       URL = {https://doi.org/10.2140/pjm.2001.201.89},
}

@article {R,
    AUTHOR = {Reidemeister, Kurt},
     TITLE = {Elementare {B}egr\"{u}ndung der {K}notentheorie},
   JOURNAL = {Abh. Math. Sem. Univ. Hamburg},
  FJOURNAL = {Abhandlungen aus dem Mathematischen Seminar der
              Universit\"{a}t Hamburg},
    VOLUME = {5},
      YEAR = {1927},
    NUMBER = {1},
     PAGES = {24--32},
      ISSN = {0025-5858,1865-8784},
   MRCLASS = {99-04},
  MRNUMBER = {3069462},
       DOI = {10.1007/BF02952507},
       URL = {https://doi.org/10.1007/BF02952507},
}

@article {kegel_2018,
    AUTHOR = {Kegel, Marc},
     TITLE = {The {L}egendrian knot complement problem},
   JOURNAL = {J. Knot Theory Ramifications},
  FJOURNAL = {Journal of Knot Theory and its Ramifications},
    VOLUME = {27},
      YEAR = {2018},
    NUMBER = {14},
     PAGES = {1850067, 36},
      ISSN = {0218-2165,1793-6527},
   MRCLASS = {53D35 (53D10 57M25 57M27 57R17 57R65)},
  MRNUMBER = {3896311},
MRREVIEWER = {Janko\ Latschev},
       DOI = {10.1142/S0218216518500670},
       URL = {https://doi.org/10.1142/S0218216518500670},
}
\vspace{10pt}
\end{document}